\documentclass[11pt]{article}
\usepackage{amsfonts, amsmath, amssymb, latexsym, eucal, amscd,cite} 
\usepackage[all]{xy}
\newtheorem{theorem}{Theorem}[section]
\newtheorem{lemma}[theorem]{Lemma}
\newtheorem{proposition}[theorem]{Proposition}
\newtheorem{corollary}[theorem]{Corollary}
\newtheorem{exAux}[theorem]{Example}

\newtheorem{Def}[theorem]{Definition}
\newenvironment{definition}{\begin{Def} \rm}{\end{Def}}
\newtheorem{Note}[theorem]{Note}
\newenvironment{note}{\begin{Note} \rm}{\end{Note}}
\newtheorem{Problem}[theorem]{Problem}

\newtheorem{Rem}[theorem]{Remark}

\newtheorem{Not}[theorem]{Notation}

\newtheorem{Conj}[theorem]{Conjecture}

\newtheorem{Ass}[theorem]{Assumption}

\newenvironment{proof}{\medskip\noindent{\bf Proof.\ }}{\qed\medskip}
\newenvironment{proofof}[1]{\medskip\noindent{\bf Proof  of {#1}.\ 
}}{\qed\medskip}
\newcommand{\qed}{\hfill\mbox{$\Box$\qquad\qquad}}

\newcommand{\F}{\mathbb{F}}

\newcommand{\vphi}{\varphi}
\renewcommand{\th}{\theta}

\newcommand{\vth}{\vartheta}

\newcommand{\tth}{\tilde{\theta}}
\newcommand{\tvphi}{\tilde{\varphi}}
\newcommand{\tphi}{\tilde{\phi}}

\newcommand{\tq}{\tilde{q}}
\newcommand{\tvth}{\tilde{\vartheta}}

\newif\ifDRAFT
\DRAFTfalse

\begin{document}

\begin{center}
\LARGE\bf
The end-parameters of a Leonard pair
\end{center}

\begin{center}
\Large
Kazumasa Nomura
\end{center}

\bigskip

{\small
\begin{quote}
\begin{center}
{\bf Abstract}
\end{center}
Fix an algebraically closed field $\F$ and an integer $d \geq 3$.
Let $V$ be a vector space over $\F$ with dimension $d+1$.
A Leonard pair on $V$ is a pair of diagonalizable linear transformations
$A: V \to V$ and $A^* : V \to V$,
each acting in an irreducible tridiagonal fashion on an eigenbasis
for the other one.
There is an object related to a Leonard pair called a Leonard system.
It is known that a Leonard system is determined up to isomorphism
by a sequence of scalars
$(\{\th_i\}_{i=0}^d, \{\th^*_i\}_{i=0}^d, \{\vphi_i\}_{i=1}^d, \{\phi_i\}_{i=1}^d)$,
called its parameter array.
The scalars $\{\th_i\}_{i=0}^d$ (resp.\ $\{\th^*_i\}_{i=0}^d$) are mutually
distinct,
and the expressions
$(\th_{i-2} - \th_{i+1})/(\th_{i-1}-\th_{i})$,
$(\th^*_{i-2} - \th^*_{i+1})/(\th^*_{i-1}-\th^*_{i})$
are equal and independent of $i$ for $2 \leq i \leq d-1$.
Write this common value as $\beta+1$.
In the present paper, we consider the ``end-parameters''
$\th_0$, $\th_d$, $\th^*_0$, $\th^*_d$, $\vphi_1$, $\vphi_d$,
$\phi_1$, $\phi_d$ of the parameter array.
We show that a Leonard system is determined up to isomorphism 
by the end-parameters and $\beta$.
We display a relation between the end-parameters and $\beta$.
Using this relation, we show that there are up to isomorphism at most $\lfloor (d-1)/2 \rfloor$ 
Leonard systems that have specified end-parameters.
The upper bound $\lfloor (d-1)/2 \rfloor$ is best possible.
\end{quote}
}

\section{Introduction}

Throughout the paper $\F$ denotes an algebraically closed field.

We begin by recalling the notion of a Leonard pair.
We use the following terms.
A square matrix is said to be {\em tridiagonal} whenever each nonzero
entry lies on either the diagonal, the subdiagonal, or the superdiagonal.
A tridiagonal matrix is said to be {\em irreducible} whenever
each entry on the subdiagonal is nonzero and each entry on the superdiagonal is nonzero.

\begin{definition}  {\rm (See \cite[Definition 1.1]{T:Leonard}.)}    \label{def:LP}   \samepage
\ifDRAFT {\rm def:LP}. \fi
Let $V$ be a vector space over $\F$ with finite positive dimension.
By a {\em Leonard pair on $V$} we mean an ordered pair of linear transformations
$A : V \to V$ and $A^*: V \to V$ that satisfy (i) and (ii) below:
\begin{itemize}
\item[\rm (i)]
There exists a basis for $V$ with respect to which the matrix representing
$A$ is irreducible tridiagonal and the matrix representing $A^*$ is diagonal.
\item[\rm (ii)]
There exists a basis for $V$ with respect to which the matrix representing
$A^*$ is irreducible tridiagonal and the matrix representing $A$ is diagonal.
\end{itemize}
\end{definition}

\begin{note}    \samepage
According to a common notational convention, $A^*$ denotes the
conjugate transpose of $A$.
We are not using this convention.
In a Leonard pair $A,A^*$ the matrices $A$ and $A^*$ are arbitrary subject to
the conditions (i) and (ii) above.
\end{note}

We refer the reader to \cite{NT:balanced,T:Leonard,T:24points,T:array,T:survey}
for background on Leonard pairs.

For the rest of this section, 
fix an integer $d \geq 0$ and a vector space $V$ over $\F$ with
dimension $d+1$.
Consider a Leonard pair $A,A^*$ on $V$.
By \cite[Lemma 1.3]{T:Leonard} each of $A,A^*$ has mutually distinct $d+1$ eigenvalues.
Let $\{\th_i\}_{i=0}^d$ be an ordering of the eigenvalues of $A$,
and let $\{V_i\}_{i=0}^d$ be the corresponding eigenspaces.
For $0 \leq i \leq d$ define $E_i : V \to V$ such that
$(E_i - I) V_i = 0$ and $E_i V_j=0$ for $j \neq i$ $(0 \leq j \leq d)$.
Here $I$ denotes the identity.
We call $E_i$ the {\em primitive idempotent} of $A$ associated with $\th_i$.
The primitive idempotent $E^*_i$ of $A^*$ associated with $\th^*_i$
is similarly defined.
For $0 \leq i \leq d$ pick a nonzero $v_i \in V_i$.
Note that $\{v_i\}_{i=0}^d$ is a basis for $V$.
We say the ordering $\{E_i\}_{i=0}^d$ is {\em standard} whenever
the basis $\{v_i\}_{i=0}^d$ satisfies Definition \ref{def:LP}(ii).
A standard ordering of the primitive idempotents of $A^*$ is
similarly defined.
For a standard ordering $\{E_i\}_{i=0}^d$, the ordering $\{E_{d-i}\}_{i=0}^d$
is also standard and no further ordering is standard.
Similar result applies to a standard ordering of the primitive idempotents
of $A^*$.

\begin{definition}   {\rm (See \cite[Definition 1.4]{T:Leonard}.) }    \label{def:LS}
\ifDRAFT {\rm def:LS}. \fi 
By a {\em Leonard system} on $V$ we mean a sequence
\begin{equation}   \label{eq:Phi}
 \Phi = (A, \{E_i\}_{i=0}^d, A^*, \{E^*_i\}_{i=0}^d),
\end{equation}
where $A,A^*$ is a Leonard pair on $V$, and 
$\{E_i\}_{i=0}^d$ (resp.\ $\{E^*_i\}_{i=0}^d$) is a standard ordering 
of the primitive idempotents of $A$ (resp.\ $A^*$).
We say {\em $\Phi$ is over $\F$}.
We call $d$ the {\em diameter} of $\Phi$.
\end{definition}

We recall the notion of an isomorphism of Leonard systems.
Consider a Leonard system \eqref{eq:Phi} on $V$
and a Leonard system $\Phi' = (A', \{E'_i\}_{i=0}^d, A^{*\prime}, \{E^{* \prime}_i\}_{i=0}^d)$
on a vector space $V'$ with dimension $d+1$.
By an {\em isomorphism of Leonard systems from $\Phi$ to $\Phi'$} we mean
a linear bijection $\sigma : V \to V$
such that $\sigma A = A' \sigma$, $\sigma A^* = A^{*\prime} \sigma$,
and $\sigma E_i = E'_i \sigma$, $\sigma E^*_i = E^{*\prime} \sigma$ for $0 \leq i \leq d$.
Leonard systems $\Phi$ and $\Phi'$ are said to be {\em isomorphic}
whenever there exists an isomorphism of Leonard systems from
$\Phi$ to $\Phi'$.

For a Leonard system \eqref{eq:Phi} over $\F$,
each of the following is a Leonard system over $\F$:
\begin{align*}
\Phi^{*}  &:= (A^*, \{E^*_i\}_{i=0}^d, A, \{E_i\}_{i=0}^d), 
\\
\Phi^{\downarrow} &:= (A, \{E_i\}_{i=0}^d, A^*, \{E^*_{d-i}\}_{i=0}^d),
\\
\Phi^{\Downarrow} &:= (A, \{E_{d-i}\}_{i=0}^d, A^*, \{E^*_{i}\}_{i=0}^d).
\end{align*}
Viewing $*$, $\downarrow$, $\Downarrow$ as permutations on the set of
all the Leonard systems,
\begin{equation}    \label{eq:relation}
*^2 = \; \downarrow^2 \; = \; \Downarrow^2 = 1,   \qquad
\Downarrow * = * \downarrow,  \qquad
\downarrow * = * \Downarrow,  \qquad
\downarrow \Downarrow \; = \; \Downarrow \downarrow.
\end{equation}
The group generated by symbols $*$, $\downarrow$, $\Downarrow$ subject
to the relations \eqref{eq:relation} is the dihedral group $D_4$. 
We recall $D_4$ is the group of symmetries of a square, and has $8$ elements.
For an element $g \in D_4$ and for an object $f$ associated with $\Phi$, 
let $f^g$ denote the corresponding object associated with $\Phi^{g^{-1}}$.

We recall the notion of a parameter array.

\begin{definition}   {\rm (See \cite[Section 2]{T:array}, \cite[Theorem 4.6]{NT:formula}.) }
\label{def:parray}
\ifDRAFT {\rm def:parray}. \fi
Consider a Leonard system \eqref{eq:Phi} over $\F$.
By the {\em parameter array of $\Phi$} we mean the sequence
\begin{equation}            \label{eq:parray}
   (\{\th_i\}_{i=0}^d, \{\th^*_i\}_{i=0}^d, \{\vphi_i\}_{i=1}^d, \{\phi_i\}_{i=1}^d),
\end{equation}
where $\th_i$ is the eigenvalue of $A$ associated with $E_i$,
$\th^*_i$ is the eigenvalue of $A^*$ associate with $E^*_i$,
and
\begin{align*}
 \vphi_i &=  (\th^*_0 - \th^*_i)
       \frac{\text{tr}(E^*_0 \prod_{h=0}^{i-1}(A-\th_h I))}
               {\text{tr}(E^*_0\prod_{h=0}^{i-2}(A-\th_h I))},
\\
 \phi_i &= (\th^*_0 - \th^*_i)
       \frac{\text{tr}(E^*_0\prod_{h=0}^{i-1}(A-\theta_{d-h}I))}
              {\text{tr}(E^*_0\prod_{h=0}^{i-2}(A-\theta_{d-h}I))}, 
\end{align*}
where tr means trace.
In the above expressions, the denominators are nonzero by 
\cite[Corollary 4.5]{NT:formula}.
\end{definition}

The following two results are fundamental in the theory of Leonard pairs.

\begin{lemma}   {\rm (See  \cite[Theorem 1.9]{T:Leonard}.) } 
\label{lem:unique}    \samepage
\ifDRAFT {\rm lem:unique}. \fi
A Leonard system is determined up to isomorphism by its
parameter array.
\end{lemma}

\begin{lemma}  {\rm (See \cite[Theorem 1.9]{T:Leonard}.) } \label{lem:classify}
\ifDRAFT {\rm lem:classify}. \fi
Consider a sequence \eqref{eq:parray} consisting of scalars taken from $\F$.
Then there exists a Leonard system $\Phi$ over $\F$ with parameter array 
\eqref{eq:parray} 
if and only if {\rm (i)--(v)} hold below:
\begin{itemize}
\item[\rm (i)]  
$\th_i \neq \th_j$, $\;\; \th^*_i \neq \th^*_j\;\;$
    $\;\;(0 \leq i < j \leq d)$.
\item[\rm (ii)]
$\vphi_i \neq 0$, $\;\; \phi_i \neq 0\;\;$ $\;\; (1 \leq i \leq d)$.
\item[\rm (iii)]
$ \displaystyle
 \vphi_i = \phi_1 \sum_{\ell=0}^{i-1} 
                            \frac{\th_\ell - \th_{d-\ell}}
                                   {\th_0 - \th_d}
             + (\th^*_i - \th^*_0)(\th_{i-1} - \th_d)  \qquad (1 \leq i \leq d).
$
\item[\rm (iv)]
$ \displaystyle
 \phi_i = \vphi_1 \sum_{\ell=0}^{i-1} 
                            \frac{\th_\ell - \th_{d-\ell}}
                                   {\th_0 - \th_d}
              + (\th^*_i - \th^*_0)(\th_{d-i+1} - \th_0)  \qquad (1 \leq i \leq d).
$
\item[\rm (v)]
The expressions
\begin{equation}    \label{eq:indep}
   \frac{\th_{i-2} - \th_{i+1}}
          {\th_{i-1}-\th_{i}},
 \qquad\qquad
   \frac{\th^*_{i-2} - \th^*_{i+1}}
          {\th^*_{i-1} - \th^*_{i}}
\end{equation}
are equal and independent of $i$ for $2 \leq i \leq d-1$.
\end{itemize}
\end{lemma}

\begin{definition}  {\rm (See \cite[Definition 1.1]{T:array}.) }  \label{def:parrayF}   \samepage
\ifDRAFT {\rm def:parrayF}. \fi
By a {\em parameter array over $\F$} we mean a sequence \eqref{eq:parray}
consisting of scalars taken from $\F$ that satisfy 
conditions {\rm (i)--(v)} in Lemma \ref{lem:classify}.
\end{definition}

\begin{definition}   \samepage    
Let $\Phi$ be a Leonard system over $\F$ with
parameter array \eqref{eq:parray}.
By the {\em fundamental parameter} of $\Phi$ (or \eqref{eq:parray})
we mean one less than the common value of \eqref{eq:indep}.
\end{definition}

The $D_4$ action affects the parameter array as follows:

\begin{lemma}   {\rm (See \cite[Theorem 1.11]{T:Leonard}.) }  \label{lem:D4}  \samepage
\ifDRAFT {\rm lem:D4}. \fi
Consider a Leonard system \eqref{eq:Phi} over $\F$ with parameter array \eqref{eq:parray}.
Then for $g \in \{\downarrow, \Downarrow, *\}$ the parameters
$\th_i^g$, ${\th^*_i}^g$, $\vphi^g_i$, $\phi^g_i$ are as follows:
\[
\begin{array}{c|cccc}
 g \; & \quad \th^g_i & \quad {\th^*_i}^g & \quad  \vphi^g_i & \; \phi^g_i 
\\ \hline
 \downarrow \; & \quad \th_i & \quad  \th^*_{d-i} &  \quad  \phi_{d-i+1} & \;  \vphi_{d-i+1} 
 \rule{0mm}{4mm}
\\
 \Downarrow \;  &\quad \th_{d-i} & \quad  \th^*_i &  \quad  \phi_i & \;  \vphi_i 
\\
 * \; & \quad \th^*_i &  \quad \th_i &  \quad  \vphi_i &  \; \phi_{d-i+1} 
\end{array}
\] 
\end{lemma}

\medskip

For the rest of this section, we assume $d \geq 3$.
The present paper is motivated by the following result:

\begin{proposition}  {\rm (See \cite[Corollary 14.1]{T:Leonard}.)}  \label{prop:T}
\ifDRAFT {\rm prop:T}. \fi
Consider a Leonard system \eqref{eq:Phi} over $\F$ with parameter array
\eqref{eq:parray}.
Then the isomorphism class of $\Phi$ is determined by a sequence of $8$ 
parameters consisting of $\th_0$, $\th_1$, $\th_2$, $\th^*_0$, $\th^*_1$, $\th^*_2$, 
followed by one of $\th_3$, $\th^*_3$, followed by one of 
$\vphi_1$, $\vphi_d$, $\phi_1$, $\phi_d$.
\end{proposition}

Referring to Proposition \ref{prop:T}, observe that 
the set of the $8$ parameters is not invariant under the $D_4$ action.
Our concern is to find a $D_4$-invariant set of parameters that determines
the isomorphism class of Leonard systems. 
In the present paper, we consider the {\em end-parameters}:
\[
 \th_0, \quad
 \th_d, \quad
 \th^*_0, \quad
 \th^*_d, \quad
 \vphi_1, \quad
 \vphi_d, \quad
 \phi_1, \quad
 \phi_d.
\]
Apparently the set of the end-parameters is invariant
under the $D_4$-action.
Note that the fundamental parameter is $D_4$-invariant.

\begin{theorem}   \label{thm:main1}   \samepage
\ifDRAFT {\rm thm:main1}. \fi
A Leonard system is determined up to isomorphism
by its end-parameters and its fundamental parameter.
\end{theorem}

The end-parameters are related to the fundamental parameter
as follows:

\begin{proposition}     \label{prop:Omega}   \samepage
\ifDRAFT {\rm prop:Omega}. \fi
Consider a parameter array \eqref{eq:parray} over $\F$.
Let $\beta$ be the fundamental parameter of \eqref{eq:parray},
and pick a nonzero $q \in \F$ such that $\beta=q+q^{-1}$.
Then the scalar
\[
 \Omega = \frac{\phi_1 + \phi_d - \vphi_1 - \vphi_d}
                      {(\th_0-\th_d)(\th^*_0 - \th^*_d)}.
\]
is as follows:
\[
\begin{array}{lll|c}
 & \text{\rm \;\;\; Case } & & \Omega 
\\ \hline
\beta \neq 2,  &  \beta \neq -2 & 
 & \displaystyle  \frac{ (q-1)(q^{d-1}+1)}{q^d-1}
  \rule{0mm}{7mm}
\\
\beta = 2, &  \text{\rm Char}(\F) \neq 2  & 
 &   2/d   \rule{0mm}{5mm}
\\
\beta = -2,   &   \text{\rm Char}(\F) \neq 2,   &  \text{\rm $d$ is even}
&  2(d-1)/d   \rule{0mm}{5mm}
\\
\beta = -2,  & \text{\rm Char}(\F) \neq 2, &  \text{\rm $d$ is odd}
  &  2     \rule{0mm}{5mm}
\\
\beta = 0,  & \text{\rm Char}(\F)=2 &
&  0     \rule{0mm}{5mm}
\end{array}
\]
\end{proposition}

\begin{corollary}    \label{cor:Omega}    \samepage
\ifDRAFT {\rm cor:Omega}. \fi
With reference to Proposition \ref{prop:Omega}, 
$\Omega \neq 1$.
\end{corollary}

\begin{theorem}   \label{thm:main2}   \samepage
\ifDRAFT {\rm thm:main2}. \fi
There exist up to isomorphism at most $\lfloor (d-1)/2 \rfloor$ Leonard systems
with diameter $d$ that have specified end-parameters.
\end{theorem}

In Theorem \ref{thm:main2}, the upper bound $\lfloor (d-1)/2 \rfloor$ is best possible:

\begin{theorem}   \label{thm:main3}   \samepage
\ifDRAFT {\rm thm:main3}. \fi
Assume $\text{\rm Char}(\F) \neq 2$ and $d$ does not vanish in $\F$.
Then there exist mutually non-isomorphic $\lfloor (d-1)/2 \rfloor$ Leonard systems 
with diameter $d$ that have common end-parameters.
\end{theorem}

The paper is organized as follows.
In Section \ref{sec:types} we recall some formulas concerning the parameter array.
In Section \ref{sec:proof1} we prove Theorem \ref{thm:main1}.
In Section \ref{sec:proofprop} we prove Proposition \ref{prop:Omega}.
In Section \ref{sec:polynomial} we consider a certain polynomial which is
used in the proof of Theorems \ref{thm:main2} and \ref{thm:main3}.
In Section \ref{sec:proofmain2} we prove Theorem \ref{thm:main2}.
In Section \ref{sec:const} we try to construct a parameter array
having specified end-parameters.
In Section \ref{sec:proofmain3} we prove Theorem \ref{thm:main3}.
In Appendix we display formulas that represent $\{\vphi_i\}_{i=1}^d$ and
$\{\phi_i\}_{i=1}^d$ in terms of the end-parameters and the fundamental parameter.

\section{Parameter arrays in closed form}
\label{sec:types}

Fix an integer $d \geq 3$.
Let \eqref{eq:parray} be a parameter array over $\F$ with
fundamental parameter $\beta$.
We consider the following types of the parameter array:
\[
 \begin{array}{c|lll}
  \text{Type} & & \text{Description} 
\\ \hline
  \text{I} & \;\; \beta \neq  2, & \; \beta \neq -2 \rule{0mm}{4.5mm}
\\
 \text{II} & \;\; \beta=2, & \text{Char}(\F) \neq 2
\\
\text{III}^+ & \;\; \beta = -2, & \text{Char}(\F) \neq 2, & \text{$d$ is even}
\\
\text{III}^- & \;\; \beta = -2, &  \text{Char}(\F) \neq 2, & \text{$d$ is odd}
\\
\text{IV} & \;\; \beta = 0, & \text{Char}(\F)=2
 \end{array}
\]
For each type we display formulas that represent the parameter array
in closed form.

\begin{lemma}  {\rm (See \cite[Lemma 13.1]{NT:affine}.) }  \label{lem:typeIclosed}
\ifDRAFT {\rm lem:typeIclosed}. \fi
Assume the parameter array \eqref{eq:parray} has type  I.
Pick a nonzero $q \in \F$ such that $\beta = q+q^{-1}$.
Then there exist scalars $\eta$, $h$, $\mu$, $\eta^*$, $h^*$, $\mu^*$, $\tau$ in $\F$
such that
\begin{align*}
 \th_i &= \eta + \mu q^i + h q^{d-i},
\\
 \th^*_i &= \eta^* + \mu^* q^i + h^* q^{d-i}
\intertext{for $0 \leq i \leq d$, and}
 \vphi_i &= (q^i-1)(q^{d-i+1}-1)(\tau - \mu \mu^* q^{i-1} - h h^* q^{d-i}),
\\
 \phi_i &= (q^i-1)(q^{d-i+1}-1)(\tau - h \mu^* q^{i-1} - \mu h^* q^{d-i})
\end{align*}
for $1 \leq i \leq d$.
\end{lemma}

\begin{note}   \label{note:I}    \samepage
\ifDRAFT {\rm note:I}. \fi
With reference to Lemma \ref{lem:typeIclosed},
for $1 \leq i \leq d$ we have $q^i \neq 1$; otherwise $\vphi_i=0$.
\end{note}

\begin{lemma}  {\rm (See \cite[Lemma 14.1]{NT:affine}.) }  \label{lem:typeIIclosed}   \samepage
\ifDRAFT {\rm lem:typeIIclosed}. \fi
Assume the parameter array \eqref{eq:parray} has type II.
Then there exist scalars $\eta$, $h$, $\mu$, $\eta^*$, $h^*$, $\mu^*$, $\tau$ in $\F$
such that
\begin{align*}
 \th_i &= \eta + \mu (i-d/2) + h i(d-i),
\\
 \th^*_i &= \eta^* + \mu^* (i-d/2) + h^* i(d-i)
\intertext{for $0 \leq i \leq d$, and}
 \vphi_i &= i(d-i+1)(\tau- \mu\mu^*/2+(h \mu^* + \mu h^*)(i-(d+1)/2)+h h^*(i-1)(d-i)),
\\
 \phi_i &= i(d-i+1)(\tau+ \mu\mu^*/2+(h \mu^* - \mu h^*)(i-(d+1)/2)+h h^*(i-1)(d-i))
\end{align*}
for $1 \leq i \leq d$.
\end{lemma}

\begin{note}   \label{note:II}   \samepage
\ifDRAFT {\rm note:II}. \fi
With reference to Lemma \ref{lem:typeIIclosed},
$\text{Char}(\F) \neq i$ for any prime $i \leq d$;
otherwise $\vphi_i = 0$.
\end{note}

\begin{lemma}  {\rm (See \cite[Lemma 15.1]{NT:affine}.) }  \label{lem:typeIII+closed}  \samepage
\ifDRAFT {\rm lem:typeIII+closed}. \fi
Assume the parameter array \eqref{eq:parray} has type III$^+$.
Then there exist scalars $\eta$, $h$, $s$, $\eta^*$, $h^*$, $s^*$, $\tau$ in $\F$
such that
\begin{align*}
 \th_i &=
   \begin{cases}
       \eta + s + h (i-d/2) & \text{ if $i$ is even},
   \\
       \eta - s - h(i-d/2) & \text{ if $i$ is odd},
   \end{cases}
\\
 \th^*_i &=
   \begin{cases}
     \eta^* + s^* + h^* (i-d/2) & \text{ if $i$ is even},
   \\
      \eta^* - s^*  - h^* (i-d/2) & \text{ if $i$ is odd}
   \end{cases}
\intertext{for $0 \leq i \leq d$, and}
 \vphi_i &=
  \begin{cases}
    i \big( \tau-s h^* - s^* h - h h^* (i-(d+1)/2) \big) & \text{ if $i$ is even},
  \\
   (d-i+1) \big( \tau + s h^* + s^* h + h h^* (i-(d+1)/2) \big) & \text{ if $i$ is odd},
  \end{cases}
\\
 \phi_i &=
  \begin{cases}
    i \big( \tau-s h^* + s^* h + h h^* (i-(d+1)/2) \big) & \text{ if $i$ is even},
  \\
   (d-i+1) \big( \tau + s h^* - s^* h - h h^* (i-(d+1)/2) \big) & \text{ if $i$ is odd}
  \end{cases}
\end{align*}
for $1 \leq i \leq d$.
\end{lemma}

\begin{note}   \label{note:III+}   \samepage
\ifDRAFT {\rm note:III+}. \fi
With reference to Lemma \ref{lem:typeIII+closed},
 $\text{Char}(\F) \neq i$ for any prime $i \leq d/2$;
otherwise $\vphi_{2i} = 0$.
By this and since $\text{Char}(\F) \neq 2$ we find $\text{Char}(\F)$ is either $0$ or
an odd prime greater than $d/2$.
Observe that neither of $d$, $d-2$ vanish in $\F$; otherwise $\text{Char}(\F)$ must divide $d/2$
or $(d-2)/2$.
\end{note}

\begin{lemma}  {\rm (See \cite[Lemma 16.1]{NT:affine}.) }  \label{lem:typeIII-closed}  \samepage
\ifDRAFT {\rm lem:typeIII-closed}. \fi
Assume the parameter array \eqref{eq:parray} has type III$^-$.
Then there exist scalars $\eta$, $h$, $s$, $\eta^*$, $h^*$, $s^*$, $\tau$ in $\F$
such that
\begin{align*}
 \th_i &=
   \begin{cases}
       \eta + s + h (i-d/2) & \text{ if $i$ is even},
   \\
       \eta - s - h(i-d/2) & \text{ if $i$ is odd},
   \end{cases}
\\ 
 \th^*_i &=
   \begin{cases}
     \eta^* + s^* + h^* (i-d/2) & \text{ if $i$ is even},
   \\
      \eta^* - s^* - h^* (i-d/2) & \text{ if $i$ is odd}
   \end{cases}
\intertext{for $0 \leq i \leq d$, and}
 \vphi_i &=
  \begin{cases}
   h h^* i (d-i+1) & \text{ if $i$ is even},
  \\
   \tau - 2 s s^* + i (d-i+1)h h^* - 2(h s^* + h^* s) \big( i - (d+1)/2 \big)
                                             & \text{ if $i$ is odd},
  \end{cases}
\\
 \phi_i &=
  \begin{cases}
      h h^* i (d-i+1) & \text{ if $i$ is even},
  \\
   \tau + 2 s s^* + i (d-i+1)h h^* - 2(h s^* - h^* s) \big( i - (d+1)/2 \big)
                                             & \text{ if $i$ is odd}
  \end{cases}
\end{align*}
for $1 \leq i \leq d$.
\end{lemma}

\begin{note}   \label{note:III-}   \samepage
\ifDRAFT {\rm note:III-}. \fi
With reference to Lemma \ref{lem:typeIII-closed},
$\text{Char}(\F) \neq i$ for any prime $i \leq d/2$;
otherwise $\vphi_{2i} = 0$.
By this and since $\text{Char}(\F) \neq 2$ we find $\text{Char}(\F)$ is either $0$ or
an odd prime greater than $d/2$.
Observe $d-1$ does not vanish in $\F$; otherwise $\text{Char}(\F)$ must divide $(d-1)/2$.
\end{note}

\begin{lemma}  {\rm (See \cite[Lemma 17.1]{NT:affine}.) }  \label{lem:typeIVclosed}  \samepage
\ifDRAFT {\rm lem:typeIVclosed}. \fi
Assume the parameter array \eqref{eq:parray} has type IV.
Then $d=3$,
and there exist scalars $h$, $s$, $h^*$, $s^*$, $r$ in $\F$
such that
\begin{align*}
 \th_1 &= \th_0 + h(s+1),  &
 \th_2 &= \th_0 + h, &
 \th_3 &= \th_0 + h s,
\\
 \th^*_1 &= \th^*_0 + h^*(s^* + 1), &
 \th^*_2 &= \th^*_0 + h^*, &
 \th^*_3 &= \th^*_0 + h^* s^*,
\\
 \vphi_1 &= h h^* r, &
 \vphi_2 &= h h^*, &
 \vphi_3 &= h h^* (r+s+s^*),
\\
 \phi_1 &= h h^* (r + s + s s^*) , &
 \phi_2 &= h h^*, &
 \phi_3 &= h h^* (r + s^* + s s^*).
\end{align*}
\end{lemma}

We mention a lemma for later use.
Pick a nonzero $q \in \F$ such that $\beta = q+q^{-1}$.

\begin{lemma}   {\rm (See \cite[Lemma 10.2]{T:Leonard}.)}  \label{lem:vth} \samepage
\ifDRAFT {\rm lem:vth}. \fi
The following hold:
\begin{itemize}
\item[\rm (i)]
Assume the parameter array \eqref{eq:parray} has type I. Then for $1 \leq i \leq d$
\[
 \sum_{\ell=0}^{i-1} \frac{\th_{\ell} - \th_{d- \ell}}{\th_0 - \th_d} 
 = \frac{(q^i-1)(q^{d-i+1}-1)}{(q-1)(q^d-1)}.
\]
\item[\rm (ii)]
Assume the parameter array \eqref{eq:parray}  has type II.  Then for $1 \leq i \leq d$
\[
  \sum_{\ell=0}^{i-1} \frac{\th_{\ell} - \th_{d- \ell}}{\th_0 - \th_d}
  = \frac{i(d-i+1)}{d}.
\]
\item[\rm (iii)]
Assume  the parameter array  \eqref{eq:parray}  has type III$^+$. Then for $1 \leq i \leq d$
\[
  \sum_{\ell=0}^{i-1} \frac{\th_{\ell} - \th_{d- \ell}}{\th_0 - \th_d}
 = 
  \begin{cases}
    i/d & \text{ if $i$ is even},
  \\
   (d-i+1)/d & \text{ if $i$ is odd}.
 \end{cases}
\]
\item[\rm (iv)]
Assume the parameter array \eqref{eq:parray}  has type III$^-$.  Then for $1 \leq i \leq d$
\[
 \sum_{\ell=0}^{i-1} \frac{\th_{\ell} - \th_{d- \ell}}{\th_0 - \th_d}
  = 
  \begin{cases}
    0 & \text{ if $i$ is even},
  \\
   1 & \text{ if $i$ is odd}.
 \end{cases}
\]
\item[\rm (iii)]
Assume the parameter array  \eqref{eq:parray} has type IV. Then for $1 \leq i \leq d$
\[
  \sum_{\ell=0}^{i-1} \frac{\th_{\ell} - \th_{d- \ell}}{\th_0 - \th_d}
 =
  \begin{cases}
    0 & \text{ if $i$ is even},
  \\
   1 & \text{ if $i$ is odd}.
 \end{cases}
\]
\end{itemize}
\end{lemma}

\section{Proof of Theorem \ref{thm:main1}}
\label{sec:proof1} 

In this section we prove Theorem \ref{thm:main1}.
Fix an integer $d \geq 3$.
Let \eqref{eq:parray} be a parameter array over $\F$ with fundamental parameter $\beta$.
Pick a nonzero $q \in \F$ such that $\beta = q+q^{-1}$.
In the following five lemmas,
we display formulas that represent $\{\th_i\}_{i=0}^d$ and $\{\th^*_i\}_{i=0}^d$
in terms of the end-parameters and $q$.
These formulas can be routinely verified using
Lemmas \ref{lem:typeIclosed}, \ref{lem:typeIIclosed}, \ref{lem:typeIII+closed},
\ref{lem:typeIII-closed}, \ref{lem:typeIVclosed}.

\begin{lemma}   \label{lem:typeIth}   \samepage
\ifDRAFT {\rm lem:typeIth}. \fi
Assume the parameter array \eqref{eq:parray} has type I.
Then for $0 \leq i \leq d$
\begin{align*}
 \th_i
 &= \th_0
    -  \frac{(q^i-1)(q^{2d-i-1}-1)(\th_0-\th_d)}
             {(q^{d-1}-1)(q^{d}-1)}
    + \frac{(q^i-1)(q^{d-i}-1)(\phi_1-\vphi_d)}
             {(q-1)(q^{d-1}-1)(\th^*_0 - \th^*_d)},
\\
 \th^*_i
 &=  \th^*_0 
    -  \frac{(q^i-1)(q^{2d-i-1}-1)(\th^*_0-\th^*_d)}
             {(q^{d-1}-1)(q^{d}-1)}
    + \frac{(q^i-1)(q^{d-i}-1)(\phi_d-\vphi_d)}
             {(q-1)(q^{d-1}-1)(\th_0 - \th_d)}. 
\end{align*}
\end{lemma}

\begin{lemma}   \label{lem:typeIIth}   \samepage
\ifDRAFT {\rm lem:typeIIth}. \fi
Assume the parameter \eqref{eq:parray}  array has type II.
Then for $0 \leq i \leq d$
\begin{align*}
 \th_i
 &= \th_0
     - \frac{i (2d-i-1)(\th_0 - \th_d)}
               {d(d-1)}
     + \frac{i (d-i) (\phi_1 - \vphi_d)}
               {(d-1)(\th^*_0 - \th^*_d)}, 
\\
 \th^*_i
 &= \th^*_0
     - \frac{i (2d-i-1)(\th^*_0 - \th^*_d)}
               {d(d-1)}
     + \frac{i (d-i) (\phi_d - \vphi_d)}
               {(d-1)(\th_0 - \th_d)}.  
\end{align*}
\end{lemma}

\begin{lemma}   \label{lem:typeIII+th}   \samepage
\ifDRAFT {\rm lem:typeIII+th}. \fi
Assume the parameter array \eqref{eq:parray}  has type III$^+$.
Then for $0 \leq i \leq d$
\begin{align*}
 \th_i &=  
  \begin{cases}
    \displaystyle
     \th_0 - \frac{i (\th_0 - \th_d)}
                       {d}
                     & \text{ if $i$ is even},
   \\  \rule{0mm}{7mm} \displaystyle
    \th_0
    - \frac{(2d-i-1)(\th_0 - \th_d)}
              {d}
    + \frac{\phi_1 - \vphi_d}
              {\th^*_0 - \th^*_d}
                    &  \text{ if $i$ is odd},
   \end{cases}    
\\
 \th^*_i  &=
  \begin{cases}
    \displaystyle
     \th^*_0 - \frac{i (\th^*_0 - \th^*_d)}
                       {d}
                     & \text{ if $i$ is even},
   \\  \rule{0mm}{7mm}  \displaystyle
    \th^*_0
    - \frac{(2d-i-1)(\th^*_0 - \th^*_d)}
              {d}
    + \frac{\phi_d - \vphi_d}
              {\th_0 - \th_d}
                    &  \text{ if $i$ is odd}.                  
   \end{cases}   
\end{align*}
\end{lemma}

\begin{lemma}   \label{lem:typeIII-th}   \samepage
\ifDRAFT {\rm lem:typeIII-th}. \fi
Assume the parameter array \eqref{eq:parray}  has type III$^-$.
Then for $0 \leq i \leq d$
\begin{align*}
 \th_i &=  
  \begin{cases}
    \displaystyle
     \th_0 - \frac{i (\th_0 - \th_d)}
                       {d-1}
     + \frac{i (\phi_1 - \vphi_d)}
               {(d-1)(\th^*_0 - \th^*_d)}
                     & \text{ if $i$ is even},
   \\ \rule{0mm}{7mm} \displaystyle
    \th_0
    - \frac{(2d-i-1)(\th_0 - \th_d)}
              {d-1}
    + \frac{(d-i)(\phi_1 - \vphi_d)}
              {(d-1)(\th^*_0 - \th^*_d)}
                    &  \text{ if $i$ is odd},
   \end{cases}   
\\
 \th^*_i  &=
  \begin{cases}
    \displaystyle
     \th^*_0 - \frac{i (\th^*_0 - \th^*_d)}
                       {d-1}
     + \frac{i (\phi_d - \vphi_d)}
               {(d-1)(\th_0 - \th_d)}
                     & \text{ if $i$ is even},
   \\  \rule{0mm}{7mm} \displaystyle
    \th^*_0
    - \frac{(2d-i-1)(\th^*_0 - \th^*_d)}
              {d-1}
    + \frac{(d-i)(\phi_d - \vphi_d)}
              {(d-1)(\th_0 - \th_d)}
                    &  \text{ if $i$ is odd}.
   \end{cases}   
\end{align*}
\end{lemma}

\begin{lemma}   \label{lem:typeIVth}   \samepage
\ifDRAFT {\rm lem:typeIVth}. \fi
Assume the parameter \eqref{eq:parray}  array has type IV.
Then
\begin{align*}
 \th_1 &= \th_0 + \frac{\phi_1 - \vphi_3}
                               {\th^*_0 - \th^*_3},
&
 \th_2 &= \th_3 + \frac{\phi_1 - \vphi_3}
                               {\th^*_0 - \th^*_3},  
\\
 \th^*_1 &= \th^*_0 + \frac{\phi_3-\vphi_3}{\th_0-\th_3},
&
 \th^*_2 &= \th^*_3 + \frac{\phi_3 - \vphi_3}{\th_0 - \th_3}.  
\end{align*}
\end{lemma}

\begin{proofof}{Theorem \ref{thm:main1}}
By Lemmas \ref{lem:typeIth}--\ref{lem:typeIVth} the scalars
$\{\th_i\}_{i=0}^d$, $\{\th^*_i\}_{i=0}^d$ are determined by
the end-parameters and $q$.
By this and Lemma \ref{lem:classify}(iii), (iv) 
the scalars $\{\vphi_i\}_{i=1}^d$, $\{\phi_i\}_{i=1}^d$
are determined by the end-parameters and $q$.
The result follows from these comments and Lemma \ref{lem:unique}. 
\end{proofof}

\section{Proof of Proposition \ref{prop:Omega}}
\label{sec:proofprop}

In this section we prove Proposition \ref{prop:Omega}.
Fix an integer $d \geq 3$.

\begin{proofof}{Proposition \ref{prop:Omega}}
First assume the parameter array  has type I.
By Lemma \ref{lem:typeIclosed},
\begin{align*}
  \th_0  &= \eta + \mu + h q^d, 
\qquad\qquad
 \th_d = \eta + \mu q^d + h,
\\
 \th^*_0 &= \eta^* + \mu^* + h^* q^d,
\qquad\;\;
 \th^*_d = \eta^* + \mu^* q^d + h^*,
\intertext{and}
 \vphi_1 &= (q-1)(q^d-1)(\tau - \mu \mu^* - h h^* q^{d-1}),
\\
 \vphi_d &= (q^d-1)(q-1)(\tau - \mu \mu^* q^{d-1} - h h^*),
\\
 \phi_1 &= (q-1)(q^d-1)(\tau - h \mu^* - \mu h^* q^{d-1}),
\\
 \phi_d &= (q^d-1)(q-1)(\tau - h \mu^* q^{d-1} - \mu h^*).
\end{align*}
So,
\begin{align*}
(\th_0 - \th_d)(\th^*_0 - \th^*_d)
 & = (q^d-1)^2 (\mu - h)(\mu^* - h^*),
\\
 \phi_1 + \phi_d - \vphi_1 -\vphi_d
 &= (q-1)(q^d-1)(q^{d-1}+1)(\mu \mu^* + h h^* - h \mu^* -  \mu h^*). 
\end{align*}
Thus
\[
 \frac{\phi_1 + \phi_d - \vphi_1 - \vphi_d}{(\th_0 - \th_d)(\th^*_0 - \th^*_d)}
  = \frac{(q-1)(q^{d-1}+1)}{q^d-1}.
\]
We have shown the result for type I.
The proof is similar for the remaining types.
\end{proofof}

\section{A polynomial}
\label{sec:polynomial}

In this section we consider a polynomial which will be used
in our proof of Theorems \ref{thm:main2} and \ref{thm:main3}.
This polynomial is related to Proposition \ref{prop:Omega} for type I.
Fix an integer $d \geq 3$.

\begin{definition}   \label{def:f}   \samepage
\ifDRAFT {\rm def:f}. \fi
For $\omega \in \F$ we define a polynomial in $x$:
\[
 f_{\omega} (x) = \omega (x^d-1) - (x-1)(x^{d-1}+1).
\]
\end{definition}

\begin{lemma}    \label{lem:f0}   \samepage
\ifDRAFT {\rm lem:f0}. \fi
For $\omega \in \F$ the following hold:
\begin{itemize}
\item[\rm (i)]
$f_\omega(1)=0$.
\item[\rm (ii)]
Assume $\omega \neq 1$.
Then $f_\omega(x)$ has degree $d$ and $f_\omega(0) \neq 0$.
\item[\rm (iii)]
Assume $d$ is even.
Then $f_\omega(-1)=0$.
\item[\rm (iv)]
Assume  $\text{\rm Char}(\F) \neq 2$, $d$ is odd, and $\omega \neq 2$.
Then $f_\omega(-1) \neq 0$.
\item[\rm (v)]
For $0 \neq q \in \F$,
if $f_{\omega} (q)=0$ then $f_{\omega} (q^{-1})=0$.
\end{itemize}
\end{lemma}

\begin{proof}
Routine verification.
\end{proof}

\begin{lemma}   \label{lem:f}
\ifDRAFT {\rm lem:f}. \fi
For $\omega \in \F$ the following hold:
\begin{itemize}
\item[\rm (i)]
We have 
$f_{\omega} (x) = (x-1) g_{\omega} (x)$,
where
\[
  g_{\omega} (x) = \omega \sum_{r=0}^{d-1} x^r - x^{d-1}-1.
\]
\item[\rm (ii)]
Assume $d$ is even.
Then 
 $f_{\omega} (x) = (x-1)(x+1) g_{\omega} (x)$,
where
\[
 g_{\omega} (x) = \omega \sum_{r=0}^{(d-2)/2} x^{2r} 
          - \sum_{r=0}^{d-2} (-1)^r x^r.
\]
\item[\rm (iii)]
Assume $d$ is even and $d$ does not vanish in $\F$.
Then for $\omega = 2/d$ we have 
$f_{\omega} (x) = - (2/d) (x-1)^3 (x+1) g(x)$,
where
\[
 g (x) = \sum_{r=0}^{(d-4)/2} (r+1)(d/2-r-1) x^{2r}
          + \sum_{r=1}^{(d-4)/2} r (d/2-r-1) x^{2r-1}.
\]
\item[\rm (iv)]
Assume $d$ is odd and $d$ does not vanish in $\F$.
Then for $\omega = 2/d$ we have
$f_{\omega} (x) = - (1/d) (x-1)^3 g(x)$,
where
\[
 g(x) = \sum_{r=0}^{d-3} (r+1)(d-r-2)x^r.
\]
\item[\rm (v)]
Assume $d$ is even and $d$ does not vanish in $\F$.
Then for $\omega = 2(d-1)/d$ we have
$f_{\omega} (x) = (2/d)(x-1)(x+1)^3 g(x)$,
where
\[
  g(x) = \sum_{r=0}^{(d-4)/2} (r+1)(d/2-r-1) x^{2r} 
           - \sum_{r=1}^{(d-4)/2} r(d/2-r-1) x^{2r-1}.
\]
\item[\rm (vi)]
Assume $d$ is odd.
Then for $\omega = 2$ we have
$f_{\omega} (x) = (x-1)(x+1)^2 g(x)$,
where
\[
 g(x) = \sum_{r=0}^{(d-3)/2} x^{2r}.
\]
\end{itemize}
\end{lemma}

\begin{proof}
Routine verification.
\end{proof}

\begin{lemma}   \label{lem:equation} 
\ifDRAFT {\rm lem:equation}. \fi
For $\omega \in \F$ consider the equation $f_{\omega} (x)=0$.
\begin{itemize}
\item[\rm (i)]
Assume $d$ is odd.
Then there are at most $d-1$ roots of $f_{\omega} (x)=0$ other than $\pm 1$.
\item[\rm (ii)]
Assume $d$ is even.
Then there are at most $d-2$ roots of $f_{\omega} (x)=0$ other than $\pm 1$.
\item[\rm (iii)]
Assume $d$ is even and $d$ does not vanish in $\F$.
Then for $\omega = 2/d$ there are at most $d-4$ roots of 
$f_{\omega} (x)=0$ other than $\pm 1$.
\item[\rm (iv)]
Assume $d$ is odd and $d$ does not vanish in $\F$.
Then for $\omega = 2/d$ there are at most $d-3$ roots of 
$f_{\omega} (x)=0$ other than $\pm 1$.
\item[\rm (v)]
Assume $d$ is even and $d$ does not vanish in $\F$.
Then for $\omega = 2(d-1)/d$ there are at most $d-4$ roots of 
$f_{\omega} (x)=0$ other than $\pm 1$.
\item[\rm (vi)]
Assume $d$ is odd and $\omega = 2$.
Then there are at most $d-3$ roots of 
$f_{\omega} (x)=0$ other than $\pm 1$.
\end{itemize}
\end{lemma}

\begin{proof}
Immediate from Lemma \ref{lem:f}.
\end{proof}

\begin{lemma}   \label{lem:repeated}    \samepage
\ifDRAFT {\rm lem:repeated}. \fi
Assume $d$ does not vanish in $\F$.
Then the equation $f_{\omega} (x)=0$ has a repeated root 
for at most $d$ values of $\omega$.
\end{lemma}

\begin{proof}
If the equation $f_{\omega} (x)=0$ has a repeated root $q$,
then both $f_{\omega} (q)=0$ and $f'_{\omega} (q)=0$, 
where $f'_{\omega}$ is the derivative of $f_{\omega}$.
The equations $f_{\omega} (x)=0$ and $f'_{\omega} (x)=0$
have a common root if and only if the resultant of 
$f_{\omega} (x)$ and $f'_{\omega} (x)$
is zero (see \cite[Chap.\ IV.8]{Lang}).
The resultant of $f_{\omega} (x)$ and $f'_{\omega} (x)$ is the determinant of the
following matrix (we display the matrix for $d=5$):
{\small
\[
M_{\omega} =
\begin{pmatrix}
 \omega-1 & 1 & 0 & 0 & -1 & 1-\omega & 0 & 0 & 0
\\
 0 &  \omega-1 & 1 & 0 & 0 & -1 & 1-\omega & 0 & 0
\\
 0 &  0 &  \omega-1 & 1 & 0 & 0 & -1 & 1-\omega & 0
\\
 0 &  0 &  0 &  \omega-1 & 1 & 0 & 0 & -1 & 1-\omega
\\
 5(\omega-1) & 4 & 0 & 0 & -1 & 0 & 0 & 0 & 0
\\
 0 &  5(\omega-1) & 4 & 0 & 0 & -1 & 0 & 0 & 0
\\
 0 &  0 &  5(\omega-1) & 4 & 0 & 0 & -1 & 0 & 0
\\
 0 &  0 &  0 &  5(\omega-1) & 4 & 0 & 0 & -1 & 0 
\\
 0 &  0 &  0 &  0 &  5(\omega-1) & 4 & 0 & 0 & -1
\end{pmatrix}
\]
}
First assume $d$ is odd.
Then
\[
 \text{det} (M_{\omega}) =
    (\omega-1)\, (\omega-2)\, (d\, \omega-2)^3 \, \psi_1(\omega)^2,
\]
where $\psi_1(x)$ is a polynomial in $x$ with leading term
$d^{(d-3)/2} x^{d-3}$. 
Thus there are at most $d$ values of $\omega$ such that 
$\text{det}(M_{\omega})=0$.
Next assume $d$ is even.
Then
\[
  \text{det}(M_{\omega}) =
     (1 - \omega)\, (d\, \omega-2)^3\, (d\, \omega - 2(d-1))^3\, \psi_2(\omega)^2,
\]
where $\psi_2(x)$ is a polynomials in $x$ with leading term
$d^{(d-6)/2} x^{d-4}$.
Thus there are at most $d-1$ values of $\omega$ such that
$\text{det}(M_{\omega})=0$.
The result follows.
\end{proof}

\begin{lemma}    \label{lem:choiceOmega}   \samepage
\ifDRAFT {\rm lem:choiceOmega}. \fi
For $3 \leq r \leq d$ let $\Gamma_r$ denote the
set consisting of the $r$th roots of unity other than $\pm 1$:
\[
  \Gamma_r = \{ q \in \F \,|\, q^r=1, \; q^2 \neq 1\}.
\]
Let $\Gamma$ be the union of $\Gamma_r$ for $3 \leq r \leq d$.
Then there exist infinitely many $\omega \in \F$ 
such that the equation $f_{\omega} (x)=0$ has no roots in $\Gamma$.
\end{lemma}

\begin{proof}
We claim that for any $\omega \in \F$
the equation $f_{\omega}(x)=0$ has no roots in $\Gamma_d$.
Suppose $f_{\omega}(q)=0$ for some $q \in \Gamma_d$.
Then $0 = f_{\omega}(q) = q^{d-1} - q$, so $q^{d-2}=1$. 
By this and $q^d = 1$ we must have $q^2=1$, a contradiction.
Thus the claim holds.
For $q \in \Gamma \setminus \Gamma_d$ define
\[
   \omega_q = \frac{(q-1)(q^{d-1}+1)}{q^d-1},
\]
and consider the set
\[
   \Delta = \{ \omega_q \,|\, q \in \Gamma \setminus \Gamma_d\}.
\]
Note that $\F \setminus \Delta$ has infinitely many elements,
since $\F$ is infinite and $\Delta$ is finite.
For $\omega \in \F \setminus \Delta$,
the equation $f_{\omega}(x)=0$ has no roots in $\Gamma \setminus \Gamma_d$.
By this and the above claim, the equation $f_{\omega}(x)=0$ has no roots
in $\Gamma$.
The result follows.
\end{proof}

\begin{corollary}   \label{cor:Omegaexists}   \samepage
\ifDRAFT {\rm cor:Omegaexists}. \fi
Assume $d$ does not vanish in $\F$.
Then there exist infinitely many $\omega \in \F$ that satisfy both
{\rm (i)} and {\rm (ii)} below:
\begin{itemize}
\item[\rm (i)]
The equation $f_\omega(x)=0$ has no repeated roots.
\item[\rm (ii)]
The equation $f_\omega(x)=0$ has no roots in $\Gamma$,
where $\Gamma$ is from Lemma \ref{lem:choiceOmega}.
\end{itemize}
\end{corollary}

\begin{proof}
Follows from Lemmas \ref{lem:repeated} and \ref{lem:choiceOmega}.
\end{proof}

\begin{lemma}    \label{lem:choiceq}   \samepage
\ifDRAFT {\rm lem:choiceq}. \fi
Let $\omega \in \F$ with $\omega \neq 1$, $\omega \neq 2$.
Assume that the equation $f_\omega(x)=0$ has no repeated roots.
\begin{itemize}
\item[\rm (i)]
Assume $\text{\rm Char}(\F) \neq 2$ and $d$ is odd.
Then the equation $f_\omega(x)=0$ has mutually distinct $d-1$ 
nonzero roots other than $\pm 1$.
\item[\rm (ii)]
Assume $d$ is even.
Then the equation $f_\omega(x)=0$ has mutually distinct $d-2$ 
nonzero roots other than $\pm 1$. 
\end{itemize}
\end{lemma}

\begin{proof}
We claim that the equation $f_\omega(x)=0$ has mutually distinct
$d$ nonzero roots.
By Lemma \ref{lem:f0}(ii) and since $\omega \neq 1$,
the polynomial $f_\omega(x)$ has degree $d$ and $f_\omega(0) \neq 0$.
Now the claim holds by this and since $f_\omega(x)=0$ has no repeated roots.

(i):
By Lemma \ref{lem:f0}(i) $f_\omega(1)=0$.
We have $f_\omega(-1) \neq 0$ by Lemma \ref{lem:f0}(iv) and since
$\omega \neq 2$, $\text{Char}(\F) \neq 2$.
By these comments and the claim, the equation $f_\omega(x)=0$ has
mutually distinct $d-1$ nonzero roots other than $\pm 1$.

(ii):
By Lemma \ref{lem:f0}(i), (iii) each of $1$, $-1$ is a root of $f_\omega(x)=0$.
By this and the claim, the equation $f_\omega(x)=0$ has
mutually distinct $d-2$ nonzero roots other than $\pm 1$.
\end{proof}

\section{Proof of Theorem \ref{thm:main2}}
\label{sec:proofmain2}

\begin{proofof}{Theorem \ref{thm:main2}}
Suppose we are given a parameter array over $\F$:
\[
  (\{\th_i\}_{i=0}^d, \{\th^*_i\}_{i=0}^d, \{\vphi_i\}_{i=1}^d, \{\phi_i\}_{i=1}^d).
\]
Let $\tilde{P}$ denote the set of parameter arrays
\[
  \tilde{p} = (\{\tth_i\}_{i=0}^d, \{\tth^*_i\}_{i=0}^d, \{\tvphi_i\}_{i=1}^d, \{\tphi_i\}_{i=1}^d)
\]
over $\F$ that satisfy
\[
\begin{array}{ccccccc}
  \tth_0 = \th_0, & &
  \tth_d = \th_d, & &
  \tth^*_0 = \th^*_0, & &
  \tth^*_d = \th^*_d, 
\\
  \tvphi_1 = \vphi_1, & &
  \tvphi_d = \vphi_d, & &
  \tphi_1 = \phi_1,  & &
  \tphi_d = \phi_d.
\end{array}
\]
We count the number of elements of $\tilde{P}$.
By Theorem \ref{thm:main1} a parameter array in $\tilde{P}$ is determined by its
fundamental parameter.
Let $\tilde{Q}$ denote the set of nonzero $\tq \in \F$ such that $\tq + \tq^{-1}$
is the fundamental parameter for some $\tilde{p} \in \tilde{P}$.
Note that $\tq$ is determined up to inverse by the fundamental parameter.
So we count the number of elements of $\tilde{Q}$ up to inverse.
Define
\begin{align*}
 \Omega &=
   \frac{\phi_1 + \phi_d - \vphi_1 - \vphi_d}
          {(\th_0 - \th_d)(\th^*_0 - \th^*_d)}.
\end{align*}
By Proposition \ref{prop:Omega}, for $\tilde{p} \in \tilde{P}$ we obtain the equation:
\begin{equation}      \label{eq:tOmega}
\begin{array}{c|c}
\text{\rm Type of $\tilde{p}$ } & \text{\rm Equation}
\\ \hline 
\text{\rm I}
& \displaystyle  \frac{ (\tq-1)(\tq^{d-1}+1)}{\tq^d-1} = \Omega \rule{0mm}{7mm}
\\
\text{\rm II}
 &   2/d = \Omega   \rule{0mm}{5mm}
\\
\text{\rm III$^+$}
&  2(d-1)/d =\Omega  \rule{0mm}{5mm}
\\
\text{\rm  III$^-$}
  &  2 =\Omega    \rule{0mm}{5mm}
\\
\text{\rm IV}
&  0 = \Omega    \rule{0mm}{5mm}
\end{array}
\end{equation}
where $\tq + \tq^{-1}$ is the fundamental parameter of $\tilde{p}$.

We claim that at least one of $1$, $-1$ is not contained $\tilde{Q}$
when $\text{Char}(\F) \neq 2$.
By way of contradiction, assume $\text{Char}(\F) \neq 2$ and
$\{1, -1\} \subseteq \tilde{Q}$.
Then there is a $\tilde{p}_1 \in \tilde{P}$ (resp.\ $\tilde{p}_2 \in \tilde{P}$) 
that has fundamental parameter $2$ (resp.\ $-2$).
Note that $\tilde{p}_1$ has type II and $\tilde{p}_2$ has type III$^+$ or III$^-$.
So by \eqref{eq:tOmega} 
$d\, \Omega =2$, and either $d\, \Omega =2(d-1)$ or $\Omega =2$.
If $d\, \Omega=2$ and $d\, \Omega =2(d-1)$, then $d-2$ vanishes in $\F$.
If $d\, \Omega=2$ and $\Omega =2$, then $d-1$ vanishes in $\F$.
But, by Note \ref{note:II}, neither of $d-1$, $d-2$ vanishes in $\F$, a contradiction.
We have shown the claim.
Now we count the number of elements of $\tilde{Q}$ up to inverse.
Note that $\Omega \neq 1$ by Corollary \ref{cor:Omega}.
First assume $\Omega \neq 2$, $d \, \Omega \neq 2$, and
$d\, \Omega \neq 2(d-1)$.
By Lemma \ref{lem:equation}(i), (ii)
there are up to inverse at most $\lfloor (d-1)/2 \rfloor$ elements of $\tilde{Q}$.
Next assume $d$ is even and $d\, \Omega = 2$.
By Lemma \ref{lem:equation}(iii)
there are up to inverse at most $(d-4)/2$ elements of $\tilde{Q}$
other than $\pm 1$.
Next assume $d$ is odd and $d\, \Omega = 2$.
By Lemma \ref{lem:equation}(iv)
there are up to inverse at most $(d-3)/2$ elements of $\tilde{Q}$
other than $\pm 1$.
Next assume $d$ is even and $d\, \Omega = 2(d-1)$.
By Lemma \ref{lem:equation}(v)
there are up to inverse at most $(d-4)/2$ elements of $\tilde{Q}$
other than $\pm 1$.
Next assume $d$ is odd and $\Omega = 2$.
By Lemma \ref{lem:equation}(vi) 
there are up to inverse at most $(d-3)/2$ elements of $\tilde{Q}$
other than $\pm 1$.
By these comments and the claim, there are up to inverse 
at most $\lfloor (d-1)/2 \rfloor$ elements of $\tilde{Q}$.
The result follows.
\end{proofof}

\section{How to construct a parameter array having specified end-parameters}
\label{sec:const}

In this section we try to construct a parameter array
having specified end-parameters.
To simplify our description, we restrict our attention to type I;
we can proceed in a similar way for the other types.
Fix an integer $d \geq 3$,
and pick scalars 
\[
 \th_0, \quad
 \th_d, \quad
 \th^*_0, \quad
 \th^*_d, \quad
 \vphi_1, \quad
 \vphi_d, \quad
 \phi_1, \quad
 \phi_d
\]
in $\F$ such that $\th_0 \neq \th_d$ and $\th^*_0 \neq \th^*_d$.
We will try to construct a parameter array
\[
  (\{\tth_i\}_{i=0}^d, \{\tth^*_i\}_{i=0}^d, \{\tvphi_i\}_{i=1}^d, \{\tphi_i\}_{i=1}^d)
\]
that satisfies
\begin{equation}    \label{eq:coincide2}
\begin{array}{ccccccc}
  \tth_0 = \th_0,   &  &
  \tth_d = \th_d,   &  &
  \tth^*_0 = \th^*_0,  &  &
  \tth^*_d = \th^*_d, 
\\
  \tvphi_1 = \vphi_1, & &
  \tvphi_d = \vphi_d, & &
  \tphi_1 = \phi_1, & &
  \tphi_d = \phi_d.
\end{array}
\end{equation}
Define
\[
  \Omega = \frac{\phi_1 + \phi_d - \vphi_1 - \vphi_d}
                       {(\th_0 - \th_d)(\th^*_0 - \th^*_d)}.
\]
In view of Note \ref{note:I} and Proposition \ref{prop:Omega},
we assume there exists a nonzero $q \in \F$ such that 
$q^i \neq 1$ for $1 \leq i \leq d$, and
\begin{equation}         \label{eq:Omega2}
 \Omega = \frac{(q-1)(q^{d-1}+1)}{q^d-1}.
\end{equation}
In view of Lemma \ref{lem:typeIth},
we define scalars $\{\tth_i\}_{i=0}^d$, $\{\tth^*_i\}_{i=0}^d$
as follows.

\begin{definition}    \label{def:tth}   \samepage
\ifDRAFT {\rm def:tth}. \fi
For $0 \leq i \leq d$ define 
\begin{align*}
 \tth_i
 &= \th_0
    -  \frac{(q^i-1)(q^{2d-i-1}-1)(\th_0-\th_d)}
             {(q^{d-1}-1)(q^{d}-1)}
    + \frac{(q^i-1)(q^{d-i}-1)(\phi_1-\vphi_d)}
             {(q-1)(q^{d-1}-1)(\th^*_0 - \th^*_d)}, 
\\
 \tth^*_i
 &=  \th^*_0 
    -  \frac{(q^i-1)(q^{2d-i-1}-1)(\th^*_0-\th^*_d)}
             {(q^{d-1}-1)(q^{d}-1)}
    + \frac{(q^i-1)(q^{d-i}-1)(\phi_d-\vphi_d)}
             {(q-1)(q^{d-1}-1)(\th_0 - \th_d)}. 
\end{align*}
\end{definition}

The following two lemmas can be routinely verified.

\begin{lemma}   \label{lem:th0thdcoincide}
\ifDRAFT {\rm lem:th0thdcoincide}. \fi
With reference to Definition \ref{def:tth},
\[
\tth_0 = \th_0,  \qquad
\tth_d = \th_d, \qquad
\tth^*_0 = \th^*_0,  \qquad
\tth^*_d = \th^*_d.
\]
\end{lemma}

\begin{lemma}   \label{lem:tthindep}   \samepage
\ifDRAFT {\rm lem:tthindep}. \fi
Assume
 $\tth_i \neq \tth_j$, $\tth^*_i \neq \tth^*_j$ for $1 \leq i < j \leq d$.
Then each of the expressions
\[
 \frac{\tth_{i-2} - \tth_{i+1}}{\tth_{i-1}-\tth_i},  \qquad\qquad\qquad
  \frac{\tth^*_{i-2} - \tth^*_{i+1}}{\tth^*_{i-1}-\tth^*_i}
\]
is equal to $q+q^{-1}+1$ for $2 \leq i \leq d-1$.
\end{lemma}

In view of Lemma \ref{lem:vth}(i) we define scalars $\{\vth_i\}_{i=1}^d$
as follows.

\begin{definition}   \label{def:vth}   \samepage
\ifDRAFT {\rm def:vth}. \fi
For $1 \leq i \leq d$ define
\[
 \vth_i = \frac{(q^i-1)(q^{d-i+1}-1)}{(q-1)(q^d-1)}.
\]
\end{definition}

In view of Lemma \ref{lem:classify}(iii), (iv),
we define scalars $\{\tvphi_i\}_{i=1}^d$, $\{\tphi_i\}_{i=1}^d$
as follows.

\begin{definition}    \label{def:tvphitphi}   \samepage
\ifDRAFT {\rm def:tvphitphi}. \fi
For $1 \leq i \leq d$ define
\begin{align*}
 \tvphi_i &= \phi_1 \tvth_i + (\tth^*_i - \tth^*_0)(\tth_{i-1} - \tth_d), 
\\
 \tphi_i &= \vphi_1 \tvth_i + (\tth^*_i - \tth^*_0)(\tth_{d-i+1} - \tth_0). 
\end{align*}
\end{definition}

\begin{lemma}     \label{lem:vphi1vphid}   \samepage
\ifDRAFT {\rm lem:vphi1vphid}. \fi
With reference to Definition \ref{def:tvphitphi},
\[
\tvphi_1 = \vphi_d,  \qquad
\tvphi_d = \vphi_d,  \qquad
\tphi_1 = \phi_1,  \qquad
\tphi_d = \phi_d.
\]
\end{lemma}

\begin{proof}
One routinely checks that
\begin{align*}
\tvphi_1 &= 
  \phi_1 + \phi_d - \vphi_d
 - \frac{(q-1)(q^{d-1}+1)(\th_0 - \th_d)(\th^*_0 - \th^*_d)}
          {q^d-1},
\\
\tvphi_d &= \vphi_d,
\\
\tphi_1 &= \vphi_1 + \vphi_d - \phi_d
  + \frac{(q-1)(q^{d-1}+1)(\th_0 - \th_d)(\th^*_0 - \th^*_d)}
          {q^d-1},
\\
\tphi_d &=  \vphi_1 + \vphi_d - \phi_1
  + \frac{(q-1)(q^{d-1}+1)(\th_0 - \th_d)(\th^*_0 - \th^*_d)}
          {q^d-1}.
\end{align*}
Now the result follows from \eqref{eq:Omega2}.
\end{proof}

\begin{proposition}    \label{prop:const}   \samepage
\ifDRAFT {\rm prop:const}. \fi
The sequence 
$\tilde{p} = (\{\tth_i\}_{i=0}^d, \{\tth^*_i\}_{i=0}^d, \{\tvphi_i\}_{i=1}^d, \{\tphi_i\}_{i=1}^d)$
is a parameter array over $\F$ if and only if
\begin{align}
 &  \tth_i \neq \tth_j,  \qquad \tth^*_i \neq \tth^*_j 
       \qquad\qquad\; (0 \leq i < j \leq d),                         \label{eq:tthdistinct}
\\
 & \tvphi_i \neq 0, \; \qquad \tphi_i \neq 0 
         \qquad\qquad\quad (1 \leq i \leq d).                     \label{eq:tvphinonzero}
\end{align}
In this case, the parameter array $\tilde{p}$ satisfies \eqref{eq:coincide2}.
\end{proposition}

\begin{proof}
The first assertion follows from Definition \ref{def:parrayF},
Lemma \ref{lem:tthindep}, and
Definition \ref{def:tvphitphi}.
The second assertion follows from Lemmas \ref{lem:th0thdcoincide} and \ref{lem:vphi1vphid}.
\end{proof}

\section{Proof of Theorem \ref{thm:main3}}
\label{sec:proofmain3}

In this section we prove Theorem \ref{thm:main3}.
Fix an integer $d \geq 3$.
Assume $\text{Char}(\F)\neq 2$ and $d$ does not vanish in $\F$.
Recall the polynomial $f_\omega(x)$ from Definition \ref{def:f}.

By Corollary \ref{cor:Omegaexists} there exists $\omega \in \F$ 
such that
\begin{itemize}
\item
$\omega \neq 1$,  $\;\omega \neq 2$;
\item
the equation $f_\omega(x)=0$ has no repeated roots;
\item
the equation $f_\omega(x)=0$ has no roots in $\Gamma$,
where $\Gamma$ is from Lemma \ref{lem:choiceOmega}.
\end{itemize}
Fix $\omega \in \F$ that satisfies the above conditions.

By Lemma \ref{lem:choiceq} there are up to inverse precisely $\lfloor (d-1)/2 \rfloor$
nonzero roots of $f_\omega(x)=0$ other than $\pm 1$.
For such a root $q$ and for $\zeta\in \F$,
we construct a sequence $\tilde{p}(q,\zeta)$ as follows.
Note that $q^i \neq 1$ for $1 \leq i \leq d$ by the construction.
Define scalars
\begin{align*}
 \th_0 &= 0,  \qquad
 \th_d = 1,  \qquad \quad
  \th^*_0 = 0, \qquad
  \th^*_d = 1,
\\
 \vphi_1 &= 1,  \qquad
 \vphi_d = -1,  \qquad
 \phi_1 = \zeta, \qquad
 \phi_d = \omega - \zeta.
\end{align*}
Observe that
\[
   \omega = \frac{\phi_1 + \phi_d - \vphi_1 - \vphi_d}{(\th_0 - \th_d)(\th^*_0-\th^*_d)}.
\]
For $0 \leq i \leq d$ define scalars $\tth_i=\tth_i(q, \zeta)$ and $\tth^*_i = \tth^*_i(q,\zeta)$
as in Definition \ref{def:tth}.
For $1 \leq i \leq d$ define scalars $\tvphi_i = \tvphi_i (q,\zeta)$ and $\tphi_i = \tphi_i (q,\zeta)$
as in Definition \ref{def:tvphitphi}.
We have constructed a sequence
\[
 \tilde{p}(q,\zeta) 
 = (\{\tth_i(q,\zeta)\}_{i=0}^d, \{\tth^*_i(q,\zeta)\}_{i=0}^d, 
   \{\tvphi_i(q,\zeta)\}_{i=1}^d, \{\tphi_i (q,\zeta)\}_{i=1}^d).
\]
The following two lemmas can be routinely verified.

\begin{lemma}    \label{lem:tthitthj}   \samepage
\ifDRAFT {\rm lem:tthitthj}. \fi
For $0 \leq i,j \leq d$
\[
 \tth_i(q,\zeta) - \tth_j(q,\zeta) =
  \frac{(q^j-q^i) Z_1(q,\zeta)}
         {(q-1)(q^{d-1}-1)(q^d-1)},
\]
where
\[
  Z_1(q,\zeta) = \zeta (q^d-1)(q^{d-i-j}-1) + q(q^{d-1}-1)(q^{d-i-j-1}-1).
\]
\end{lemma}

\begin{lemma}    \label{lem:tthsitthsj}   \samepage
\ifDRAFT {\rm lem:tthsitthsj}. \fi
For $0 \leq i,j \leq d$
\[
 \tth^*_i(q,\zeta) - \tth^*_j(q,\zeta) =
   \frac{(q^i-q^j) Z_2(q,\zeta)}
          {(q-1)(q^{d-1}-1)(q^d-1)},
\]
where
\[
 Z_2(q,\zeta) = \zeta (q^d-1)(q^{d-i-j}-1)-(q^{d-1}+1)(q^{d-i-j+1}+1)+2q^{d-i-j}(q^{i+j}+1).
\]
\end{lemma}

\begin{lemma}   \label{lem:tthfinite}   \samepage
\ifDRAFT {\rm lem:tthfinite}. \fi
For $0 \leq i < j \leq d$ the following hold:
\begin{itemize}
\item[\rm (i)]
$\tth_i(q,\zeta) = \tth_j (q,\zeta)$ holds for only one value of $\zeta$.
\item[\rm (ii)]
$\tth^*_i (q,\zeta) = \tth^*_j (q,\zeta)$ holds for only one value of $\zeta$.
\end{itemize}
\end{lemma}

\begin{proof}
(i):
Observe by Lemma \ref{lem:tthitthj} 
that $\tth_i(q,\zeta) = \tth_j(q,\zeta)$ if and only if $Z_1(q,\zeta)=0$.
First assume $q^{d-i-j}-1 = 0$.
Then
\[
   Z_1 (q,\zeta) = (1-q)(q^{d-1}-1) \neq 0.
\]
Next assume $q^{d-i-j}-1 \neq 0$.
Then $Z_1 (q,\zeta)$ is a polynomial in $\zeta$ with degree $1$.
So $Z_1 (q,\zeta) =0$ holds for only one value of $\zeta$.
The result follows.

(ii):
Similar to the proof of (i).
\end{proof}

The following two lemmas can be routinely verified.

\begin{lemma}    \label{lem:tvphi}    \samepage
\ifDRAFT {\rm lem:tvphi}. \fi
For $1 \leq i \leq d$
\[
\tvphi_i(q,\zeta) = 
  - \frac{(q^i-1)(q^{d-i+1}-1) \, Z_3(q,\zeta)}
           {(q-1)^2 (q^{d-1}-1)^2 (q^d-1)^2},
\]
where
\begin{align*}
  Z_3(q,\zeta) =& \;\zeta^2\, (q^d-1)^2(q^{i-1}-1)(q^{d-i}-1)
\\
    & \; - \zeta (q-1)(q^d-1)(q^{d-1}+1)(q^{i-1}-1)(q^{d-i}-1)
\\
   & \; - (q^{d-1}-1)(q^i-1) \big( (q^{d-1}+1)(q^{d-i+1}+1)-2 q^{d-i}(q^i+1) \big).
\end{align*}
\end{lemma}

\begin{lemma}    \label{lem:tphi}    \samepage
\ifDRAFT {\rm lem:tphi}. \fi
For $1 \leq i \leq d$
\[
\tphi_i (q,\zeta) =
 - \frac{(q^i-1)(q^{d-i+1}-1) \, Z_4(q,\zeta)}
           {(q-1)^2 (q^{d-1}-1)^2 (q^d-1)^2},
\]
where
\begin{align*}
 Z_4(q,\zeta) =& \; \zeta^2 (q^d-1)^2(q^{i-1}-1)(q^{d-i}-1)
\\
   & \; - \zeta (q-1) (q^d-1) \big( (q^{d-i}-1)(q^{d+i-2}-1) - q^{d-i}(q^{i-1}-1)^2 \big)
\\
   & \; - (q^{d-1}-1)(q^{i-1}-1) \big( (q^{d-1}+1)(q^{d-i+2}+1) - 2 q^{d-i+1} (q^{i-1}+1) \big).
\end{align*}
\end{lemma}

\begin{lemma}   \label{lem:tvphifinite}   \samepage
\ifDRAFT {\rm lem:tvphifinite}. \fi
For $1 \leq i \leq d$ the following hold:
\begin{itemize}
\item[\rm (i)]
$\tvphi_i(q,\zeta) = 0$ holds for at most two values of $\zeta$. 
\item[\rm (ii)]
$\tphi_i (q,\zeta) = 0$ holds for at most two values of $\zeta$. 
\end{itemize}
\end{lemma}

\begin{proof}
(i):
Observe by Lemma \ref{lem:tvphi} that $\tvphi_i(q,\zeta) = 0$ if and only if
$Z_3(q,\zeta)=0$.
First assume $i=1$. Then
\[
   Z_3(q,\zeta) = (1-q)(q^{d-1}-1)^2 (q^d-1) \neq 0.
\]
Next assume $i=d$. Then
\[
  Z_3 (q,\zeta) =  (q-1)(q^{d-1}-1)^2 (q^d-1) \neq 0.
\]
Next assume $i \neq 1$ and $i \neq d$.
Then $Z_3 (q,\zeta)$ is a quadratic polynomial in $\zeta$.
So $Z_3 (q,\zeta)=0$ holds for at most two values of $\zeta$.

(ii):
Observe by Lemma \ref{lem:tphi} that $\tphi_i (q,\zeta) = 0$ if and only if
$Z_4(q,\zeta)=0$.
First assume $i=1$. Then
\[
 Z_4 (q,\zeta) = \zeta (1-q) (q^{d-1}-1)^2 (q^d-1).
\]
So $Z_4 (q,\zeta) \neq 0$ unless $\zeta=0$.
Next assume $i=d$. Then
\[
 Z_4 (q,\zeta) = (q-1)(q^{d-1}-1)^2 \big( \zeta (q^d-1) - (q-1)(q^{d-1}+1) \big).
\]
So $Z_4 (q,\zeta)=0$ for only one value of $\zeta$.
Next assume $i \neq 1$ and $i \neq d$.
Then $Z_4(q,\zeta)$ is a quadratic polynomial in $\zeta$.
So $Z_4(q,\zeta)=0$ for at most two values of $\zeta$.
\end{proof}

\begin{proofof}{Theorem \ref{thm:main3}}
By Lemma \ref{lem:choiceq} there are up to inverse precisely $\lfloor (d-1)/2 \rfloor$
nonzero roots of $f_\omega(x)=0$ other than $\pm 1$.
Write these roots as $q_1,q_2,\ldots,q_n$, where $n = \lfloor (d-1)/2 \rfloor$.
For $1 \leq r \leq n$,
by Lemmas \ref{lem:tthfinite} and \ref{lem:tvphifinite}
there are only finitely many $\zeta$ such that $\tilde{p}(q_r,\zeta)$
conflicts \eqref{eq:tthdistinct} or \eqref{eq:tvphinonzero}.
Thus there exists $\zeta \in \F$ such that $\tilde{p}(q_r,\zeta)$ satisfies
both \eqref{eq:tthdistinct} and \eqref{eq:tvphinonzero} for $1 \leq r \leq n$.
Then by Proposition \ref{prop:const},
for $1 \leq r \leq n$ the sequence $\tilde{p}(q_r,\zeta)$ is a parameter array over $\F$
that satisfy
\begin{align*}
  \tth_0 (q_r,\zeta) &= \th_0,   &  
  \tth_d(q_r,\zeta)  &= \th_d,   &
  \tth^*_0(q_r,\zeta)  &= \th^*_0,  &
  \tth^*_d(q_r,\zeta)  &= \th^*_d, 
\\
  \tvphi_1(q_r,\zeta) &= \vphi_1, &
  \tvphi_d(q_r,\zeta) &= \vphi_d, &
  \tphi_1(q_r,\zeta) &= \phi_1, &
  \tphi_d(q_r,\zeta) &= \phi_d.
\end{align*}
Now the result follows by Lemma \ref{lem:classify}.
\end{proofof}

\section{Appendix}
\label{sec:appendix}

Fix an integer $d \geq 3$.
Let \eqref{eq:parray} be a parameter array over $\F$
with fundamental parameter $\beta$.
Pick a nonzero $q \in \F$ such that $\beta = q + q^{-1}$.
In this appendix, we display formulas that represent
$\vphi_i$ and $\phi_i$ in terms of the end-parameters and $q$.

Assume \eqref{eq:parray} has type I.
Then for $1 \leq i \leq d$
{\small
\begin{align*}
\vphi_i =&
 - \frac{q^{i-1}(q^i-1)(q^{d-i}-1)(q^{d-i+1}-1)(q^{2d-i-1}-1)(\th_0-\th_d)(\th^*_0-\th^*_d)}
           {(q^{d-1}-1)^2 (q^d-1)^2}
\\
 & \;\; + \frac{(q^i-1)(q^{d-i+1}-1) \big( (q^{i-1}-1)(q^{2d-i-1}-1)\vphi_d + q^{i-1}(q^{d-i}-1)^2(\phi_1+\phi_d-\vphi_d) \big)}
              {(q-1)(q^{d-1}-1)^2(q^d-1)}
\\
 &\;\; + \frac{(q^{i-1}-1)(q^i-1)(q^{d-i}-1)(q^{d-i+1}-1)(\phi_1-\vphi_d)(\phi_d-\vphi_d)}
             {(q-1)^2(q^{d-1}-1)^2(\th_0-\th_d)(\th^*_0-\th^*_d)},
\\
\phi_i =& \;
  \frac{q^{i-1}(q^i-1)(q^{d-i}-1)(q^{d-i+1}-1)(q^{2d-i-1}-1)(\th_0-\th_d)(\th^*_0-\th^*_d)}
           {(q^{d-1}-1)^2 (q^d-1)^2}
\\
 & \;\; + \frac{(q^i-1)(q^{d-i+1}-1) \big( (q^{i-1}-1)(q^{2d-i-1}-1)\phi_d + q^{i-1}(q^{d-i}-1)^2(\vphi_1+\vphi_d-\phi_d) \big)}
              {(q-1)(q^{d-1}-1)^2(q^d-1)}
\\
 & \;\; - \frac{(q^{i-1}-1)(q^i-1)(q^{d-i}-1)(q^{d-i+1}-1)(\vphi_1-\phi_d)(\vphi_d-\phi_d)}
             {(q-1)^2(q^{d-1}-1)^2(\th_0-\th_d)(\th^*_0-\th^*_d)}.
\end{align*}
}

\bigskip\noindent
Assume \eqref{eq:parray} has type II.
Then for $1 \leq i \leq d$
\begin{align*}
\vphi_i =&
   - \frac{i(d-i)(d-i+1)(2d-i-1)(\th_0-\th_d)(\th^*_0-\th^*_d)}
             {d^2 (d-1)^2}
\\
 & \;\; + \frac{i(d-i+1) \big( (i-1)(2d-i-1)\vphi_d + (d-i)^2(\phi_1+\phi_d-\vphi_d) \big)}
                   {d(d-1)^2}
\\
 &\;\; + \frac{i(i-1)(d-i)(d-i+1)(\phi_1-\vphi_d)(\phi_d-\vphi_d)}
                  {(d-1)^2 (\th_0-\th_d)(\th^*_0-\th^*_d)},
\\
\phi_i =& \;
    \frac{i(d-i)(d-i+1)(2d-i-1)(\th_0-\th_d)(\th^*_0-\th^*_d)}
             {d^2 (d-1)^2}
\\
 & \;\; + \frac{i(d-i+1) \big( (i-1)(2d-i-1)\phi_d + (d-i)^2(\vphi_1+\vphi_d-\phi_d) \big)}
                   {d(d-1)^2}
\\
 &\;\; - \frac{i(i-1)(d-i)(d-i+1)(\vphi_1-\phi_d)(\vphi_d-\phi_d)}
                  {(d-1)^2 (\th_0-\th_d)(\th^*_0-\th^*_d)}.
\end{align*}

\bigskip\noindent
Assume \eqref{eq:parray} has type III$^+$.
Then for $1 \leq i \leq d$
\begin{align*}
\vphi_i &=
 \begin{cases}
  \displaystyle
  \frac{ i \big( d \vphi_d + (d-i)(\th_0-\th_d)(\th^*_0-\th^*_d) \big)}
         {d^2}                             &  \text{ if $i$ is even},
  \\ 
  \displaystyle
  \frac{(d-i+1) \big( d(\phi_1+\phi_d-\vphi_d) - (2d-i-1)(\th_0-\th_d)(\th^*_0-\th^*_d) \big)}
         {d^2}                             & \text{ if $i$ is odd},
  \end{cases}
\\
 \phi_i &=
  \begin{cases}
     \displaystyle
    \frac{i \big( d \phi_d - (d-i)(\th_0-\th_d)(\th^*_0-\th^*_d) \big)}
           {d^2}                           & \text{ if $i$ is even},
    \\
    \displaystyle
    \frac{(d-i+1) \big( d(\vphi_1+\vphi_d-\phi_d) + (2d-i-1)(\th_0-\th_d)(\th^*_0-\th^*_d) \big) }
           {d^2}                            & \text{ if $i$ is odd}.
   \end{cases}
\end{align*}

\bigskip\noindent
Assume \eqref{eq:parray} has type III$^-$.
Then for $1 \leq i \leq d$ the following hold.
\begin{align*}
\intertext{If $i$ is even,}
\vphi_i =& \; 
  \frac{i(d-i+1) \big( \phi_1-\vphi_d - (\th_0-\th_d)(\th^*_0-\th^*_d) \big)
                    \big( \phi_d - \vphi_d -  (\th_0-\th_d)(\th^*_0-\th^*_d) \big)}
         {(d-1)^2 (\th_0-\th_d)(\th^*_0 - \th^*_d)},
\\
\intertext{and if $i$ is odd,}
 \vphi_i =& \; 
  - \frac{(d-i)(2d-i-1)(\th_0-\th_d)(\th^*_0-\th^*_d)}
            {(d-1)^2}
\\
  & \;\; + \frac{(i-1)(2d-i-1)\vphi_d + (d-i)^2 (\phi_1+\phi_d-\vphi_d)}
                    {(d-1)^2}
\\
  & \;\; + \frac{(i-1)(d-i)(\phi_1-\vphi_d)(\phi_d-\vphi_d)}
                    {(d-1)^2 (\th_0-\th_d)(\th^*_0-\th^*_d)}.
\\
\intertext{If $i$ is even,}
  \phi_i =& \; 
   - \frac{i(d-i+1) \big( \vphi_1-\phi_d + (\th_0-\th_d)(\th^*_0-\th^*_d) \big)
                    \big( \vphi_d - \phi_d +  (\th_0-\th_d)(\th^*_0-\th^*_d) \big)}
             {(d-1)^2 (\th_0-\th_d)(\th^*_0 - \th^*_d)},
\\
\intertext{and if $i$ is odd,}
 \phi_i =& \; 
    \frac{(d-i)(2d-i-1)(\th_0-\th_d)(\th^*_0-\th^*_d)}
            {(d-1)^2}
\\
  & \;\; + \frac{(i-1)(2d-i-1)\phi_d + (d-i)^2 (\vphi_1+\vphi_d-\phi_d)}
                    {(d-1)^2}
\\
  & \;\; - \frac{(i-1)(d-i)(\vphi_1-\phi_d)(\vphi_d-\phi_d)}
                    {(d-1)^2 (\th_0-\th_d)(\th^*_0-\th^*_d)}.
\end{align*}

\bigskip\noindent
Assume \eqref{eq:parray} has type IV.
Then
\begin{align*}
 \vphi_2 &= \frac{ \big( \phi_1 - \vphi_1 + (\th_0-\th_3)(\th^*_0-\th^*_3) \big)
                          \big( \phi_1 - \vphi_3 + (\th_0-\th_3)(\th^*_0-\th^*_3) \big)}
                       { (\th_0-\th_3)(\th^*_0-\th^*_3)},
\\
 \phi_2 &= \frac{ \big( \vphi_1 - \phi_1 + (\th_0-\th_3)(\th^*_0-\th^*_3) \big)
                          \big( \vphi_1 - \phi_3 + (\th_0-\th_3)(\th^*_0-\th^*_3) \big)}
                       { (\th_0-\th_3)(\th^*_0-\th^*_3)}.
\end{align*}

\section{Acknowledgments}

The author would like to thank Paul Terwilliger for giving this paper a close reading
and offering many valuable suggestions.

\bigskip\bigskip\noindent
Kazumasa Nomura\\
Tokyo Medical and Dental University\\
Kohnodai, Ichikawa, 272-0827 Japan\\
email: knomura@pop11.odn.ne.jp

\medskip\noindent
{\small
{\bf Keywords.} Leonard pair, tridiagonal pair, tridiagonal matrix.
\\
\noindent
{\bf 2010 Mathematics Subject Classification.} 05E35, 05E30, 33C45, 33D45
}

\end{document}